\documentclass{elsarticle}
\usepackage{amsfonts}

\usepackage{graphicx}
\usepackage{pstricks}
\usepackage{amscd}
\usepackage{amsmath}
\usepackage{amsxtra}

\newcommand{\strutstretchdef}{\newcommand{\strutstretch}{\vphantom{\raisebox{1pt}{$\big($}\raisebox{-1pt}{$\big($}}}}
\newtheorem{theorem}{Theorem}[section]
\newtheorem{lemma}[theorem]{Lemma}
\newproof{proof}[theorem]{Proof}
\newtheorem{proposition}[theorem]{Proposition}
\newtheorem{corollary}[theorem]{Corollary}

\newtheorem{remark}[theorem]{Remark}

\numberwithin{equation}{section}

\newlength{\struh}
\newlength{\textminustop}
\strutstretchdef
\renewcommand{\a}{\alpha}
\renewcommand{\b}{\beta}

\renewcommand{\phi}{\varphi}

\begin{document}

\begin{frontmatter}

\title{A New Necessary Condition \\ for the Hyponormality of Toeplitz Operators \\ on the Bergman Space}

\author{\v Zeljko \v Cu\v ckovi\'c}
\address{Department of Mathematics, University of Toledo, Toledo, Ohio 43606}
\ead{zeljko.cuckovic@utoledo.edu}
\ead[url]{http://math.utoledo.edu/\symbol{126}zcuckov/}
\author{Ra\'{u}l E. Curto \footnote{The second named author was partially supported by NSF Grants
DMS-0801168 and DMS-1302666.}}
\address{Department of Mathematics, The University of Iowa, Iowa City, Iowa
52242}
\ead{rcurto@math.uiowa.edu}
\ead[url]{http://www.math.uiowa.edu/\symbol{126}rcurto/}

\bigskip
\bigskip

\begin{abstract}
A well known result of C. Cowen states that, for a symbol $\varphi \in L^{\infty }, \; \varphi \equiv \bar{f}+g \;\;(f,g\in H^{2})$, the Toeplitz operator $T_{\varphi }$ acting on the Hardy space of the unit circle is hyponormal if and only if $f=c+T_{\bar{h}}g,$ for some $c\in {\mathbb C}$, $h\in H^{\infty }$, $\left\| h\right\| _{\infty}\leq 1.$ \ In this note we consider possible versions of this result in the {\it Bergman} space case. \ Concretely, we consider Toeplitz operators on the Bergman space of the unit disk, with symbols of the form $$\varphi \equiv \alpha z^n+\beta z^m +\gamma \overline z ^p + \delta \overline z ^q,$$ where $\alpha, \beta, \gamma, \delta \in \mathbb{C}$ and $m,n,p,q \in \mathbb{Z}_+$, $m < n$ and $p < q$. \ By letting $T_{\varphi}$ act on vectors of the form $$z^k+c z^{\ell}+d z^r \; \; (k<\ell<r),$$ we study the asymptotic behavior of a suitable matrix of inner products, as $k \rightarrow \infty$. \ As a result, we obtain a sharp inequality involving the above mentioned data:
$$
\left|\alpha \right|^2 n^2 + \left|\beta \right|^2 m^2 - \left|\gamma \right|^2 p^2 - \left|\delta \right|^2 q^2  \ge 2 \left|\bar \alpha \beta m n - \bar \gamma \delta p q \right|.
$$
This inequality improves a number of existing results, and it is intended to be a precursor of basic necessary conditions for joint hyponormality of tuples of Toeplitz operators acting on Bergman spaces in one or several complex variables.  

\bigskip
\bigskip
\bigskip

\end{abstract}

\bigskip
\bigskip


\end{frontmatter}


\section{Notation and Preliminaries}
\label{Intro}

A bounded operator acting on a complex, separable, infinite dimensional Hilbert space $\mathcal{H}$ is said to be normal if $T^{\ast }T=TT^{\ast }$; quasinormal if $T$ commutes with $T^{\ast }T$; subnormal if $T=N|_{\mathcal H}$, where $N$ is normal on a Hilbert space $\mathcal{K}$ which contains $\mathcal{H}$ and $N\mathcal H\subseteq \mathcal H$; hyponormal if $T^{\ast }T\geq TT^{\ast }$; and $2$-hyponormal if $(T,T^{2})$ is (jointly) hyponormal, that is 
$$
\left( 
\begin{array}{ccc}
\lbrack T^{\ast },T] & [T^{\ast 2},T] \\ 
\lbrack T^{\ast },T^{2}] & [T^{\ast 2},T^{2}]%
\end{array}%
\right) \geq 0. 
$$
Clearly,  
$$
\textrm{normal } \Rightarrow \textrm{quasinormal } \Rightarrow \textrm{subnormal } \Rightarrow 2\textrm{-hyponormal } \Rightarrow \textrm{hyponormal}.
$$
In this paper we focus primarily on the cases $H^{2}(\mathbb T)$ and $A^2(\mathbb{D})$, the Hardy space on the unit circle $\mathbb{T}$ and the Bergman space on the unit disk $\mathbb{D}$, respectively. \ For these Hilbert spaces, we look at Toeplitz operators, that is, the operators obtained by compressing multiplication operators on the respective $L^2$ spaces to the above mentioned Hilbert spaces. \ We consider possible versions, in the Bergman space context, of C. Cowen's characterization of hyponormality for Toeplitz operators on Hardy space of the unit circle. \ Concretely, we consider Toeplitz operators on the Bergman space of the unit disk, with symbols of the form $$\varphi \equiv \alpha z^n+\beta z^m +\gamma \overline z ^p + \delta \overline z ^q,$$ where $\alpha, \beta, \gamma, \delta \in \mathbb{C}$ and $m,n,p,q \in \mathbb{Z}_+$, $m < n$ and $p < q$. \ By letting $T_{\varphi}$ act on vectors of the form $$z^k+c z^{\ell}+d z^r \; \; (k<\ell<r),$$ we study the asymptotic behavior of a suitable matrix of inner products, as $k \rightarrow \infty$. \ As a result, we obtain a sharp inequality involving the above mentioned data. \ We begin with a brief survey of the known results in the Hardy space context. 

\section{The Hardy Space Case}

Let $L^2(\mathbb{T})$ denote the space of square integrable functions with respect to the Lebesgue measure on the unit circle, and let $H^2(\mathbb{T})$ denote the subspace consisting of functions with vanishing negative Fourier coefficients; equivalently, $H^2(\mathbb{T})$ is the $L^2(\mathbb{T})$-closure of the space of analytic polynomials. \ We also let $L^{\infty}(\mathbb{T})$ and $H^{\infty}(\mathbb{T})$ denote the corresponding Banach spaces of essentially bounded functions on $\mathbb{T}$. \ The orthogonal projection from $L^2(\mathbb{T})$ onto $H^2(\mathbb{T})$ will be denoted by $P$. 

Given $\varphi \in L^{\infty }(\mathbb{T})$, the Toeplitz operator with \textit{symbol} $\varphi $ is $T_{\varphi }:H^{2}(\mathbb{T}) \rightarrow H^{2}(\mathbb{T})$, given by $T_{\varphi }f:=P(\varphi f)\;\;\;(f\in H^{2}(\mathbb{T}))$. \ $T_{\varphi }$ is said to be \textit{analytic} if $\varphi \in H^{\infty }(\mathbb{T})$. \ 

P.R. Halmos's Problem 5 (\cite{Hal1}) asks whether every subnormal Toeplitz operator is either normal or analytic. \ In 1984, C. Cowen and J. Long answered this question in the negative \cite{CoL}. \ Along the way, C. Cowen obtained a characterization of hyponormality for Toeplitz operators, as follows \cite{Co}: if $\varphi \in L^{\infty }$, $\varphi =\bar{f}+g\;\;(f,g\in H^{2})$, then $
T_{\varphi }$ is hyponormal $\Leftrightarrow \ f=c+T_{\bar{h}}g$, for some $c\in {\mathbb C}$, $h\in H^{\infty }$, and $\left\| h\right\| _{\infty
}\leq 1$. \ T. Nakazi and K. Takahashi \cite{NT} later found an alternative description: For $\varphi \in L^{\infty }$, let $\mathcal E(\varphi ):=\{k\in H^{\infty }:\left\| k\right\| _{\infty }\leq 1\
and\ \varphi -k\bar{\varphi}\in H^{\infty }\}$; then $T_{\varphi }$ is hyponormal $\Leftrightarrow \ {\mathcal E}(\varphi )\neq
\emptyset$. \ (For a generalization of Cowen's result, see \cite{Gu}.) \ In this note we take a first step toward finding suitable generalizations of these results to the case of Toeplitz operators on the Bergman space over the unit disk. \ We also wish to pursue appropriate generalizations of the results on joint hyponormality of pairs of Toeplitz operators on the Hardy space, obtained in \cite{CL1} and \cite{CHL}.

At present, there is no known characterization of subnormality of Toeplitz operators in the unit circle in terms of the symbol. \ However, we do know that every $2$-hyponormal Toeplitz operator with a \textit{trigonometric} symbol is subnormal \cite{CL1}. \ Thus, a suitable intermediate goal is to find a characterization of $2$-hyponormality in terms of the symbol, perhaps using as a starting point either Cowen's or Nakazi-Takahashi's characterizations of hyponormality.

For Toeplitz operators with trigonometric symbols, the following results describe hyponormality.

\begin{proposition} \label{Zhutrig} \cite{Zhu} \ Suppose 
$$
\varphi(z) \equiv \sum_{k=0}^n a_kz^k + \overline{\sum_{k=0}^n b_kz^k},
$$
with $a_n \ne 0$. \ Let
$$
\begin{pmatrix} \bar{c}_0\\\bar{c}_1\\  
\vdots\\ \bar{c}_{n-1}
\end{pmatrix}=
\begin{pmatrix} a_1 & a_2 & \cdots & a_{n-1} & a_n\\\
                a_2 & a_3 & \cdots & a_n & 0 \\
								\vdots & \vdots & \ddots & \vdots & \vdots \\
								a_n & 0 & \cdots & 0 & 0 
\end{pmatrix}
\begin{pmatrix} b_1\\b_2\\ \vdots\\ b_n
\end{pmatrix} . 
$$
Then $T_{\varphi}$ is hyponormal if and only if $\left|\Phi_k(c_0,\cdots,c_k)\right|\le1 \; \; (0 \le k \le n-1)$, where $\Phi_k$ denotes the Schur function introduced in \cite{Schur}.
\end{proposition}

\begin{proposition} \label{trigHardy} 
Let $\varphi$ be a
trigonometric polynomial of the form 
$$
\varphi(z)=\sum_{n=-m}^{N} a_n z^ n.
$$ 
\begin{itemize}  
\item[(i)] \ (D. Farenick and W.Y. Lee \cite{FaLe}) \ If $T_\varphi$ is a hyponormal 
operator
then $m \le N$ and $\left|a_{-m}\right| \le \left|a_N\right|$.  
\item[(ii)] \ (D. Farenick and W.Y. Lee \cite{FaLe}) \ 
If $T_\varphi$ is a hyponormal operator
then $N-m\le \text{\rm rank}\, [T_ \varphi^ *, T_\varphi]\le N$.
\item[(iii)] \ (R. Curto and W.Y. Lee \cite{CL1}) \ The hyponormality of $T_\varphi$ is
independent of the particular values of the Fourier coefficients 
$a_0,a_1, \hdots, a_{N-m}$ of $\varphi$. \ Moreover,
for $T_\varphi$ hyponormal,
the rank
of the self-commutator of $T_\varphi$ is independent of
those coefficients.
\item[(iv)] \ (D. Farenick and W.Y. Lee \cite{FaLe}) \ If $m\le N$ and $\left|a_{-m}\right|=\left|a_N\right| \ne 0$, 
then $T_\varphi$ is hyponormal if and only
if the following equation holds: 
$$
\bar a_N\
\begin{pmatrix} a_{-1}\\a_{-2}\\  
\vdots\\ \vdots\\ a_{-m}
\end{pmatrix}= a_{-m}
\begin{pmatrix} \bar 
a_{N-m+1}\\ \bar a_{N-m+2}\\ \vdots\\  \vdots\\ \bar a_N
\end{pmatrix}. \; \; \; \; ({\rm hyponormal.})
$$
In this case, the rank of $[T_\varphi ^ *, T_\varphi]$ 
is $N-m$. 
\item[(v)] (D. Farenick and W.Y. Lee \cite{FaLe}) \ 
$T_\varphi$ is normal if and only if $m=N,\ \left|a_{-N}\right|=\left|a_N\right|$, and 
{\rm (hyponormal.})
holds with $m=N$.  
\end{itemize} 
\end{proposition}


\section{The Bergman Space Case}

By analogy with the case of the unit circle, let $L^{\infty }\equiv L^{\infty }(\mathbb D)$, $H^{\infty
}\equiv H^{\infty }(\mathbb D)$, $L^{2}\equiv L^{2}(\mathbb D)$ and $A^{2}\equiv
A^{2}(\mathbb D)$ denote the relevant spaces in the case of the unit disk $\mathbb{D}$. \ Similarly, let $P:L^{2}\rightarrow A^{2}$ denote the orthogonal projection onto the Bergman space. \ For $\varphi \in L^{\infty }$, the Toeplitz operator on the Bergman space with symbol $\varphi $ is 
$$T_{\varphi }:A^{2}(\mathbb D) \rightarrow A^{2}(\mathbb D),$$ given by 
$$
T_{\varphi }f:=P(\varphi f)\;\;\;(f\in A^{2}). 
$$

$T_{\varphi }$ is said to be \textit{analytic} if $\varphi \in H^{\infty }$. 

\subsection{A Revealing Example}

Let $$\varphi \equiv \bar{z}^2 + 2z.$$

On the Hardy space $H^2(\mathbb{T})$, $T_{\varphi}$ is not hyponormal, because $m=2$, $N=1$, and $m{>}N$ (see Proposition \ref{trigHardy}).

However, on the Bergman space $A^2(\mathbb D)$ $T_{\varphi}$ is hyponormal, as we now prove. \ Consider a slight variation of the symbol, that is, 

$$\varphi \equiv \bar{z}^2 + \alpha z.  \; \; \; (\alpha \in \mathbb C)$$.

Observe that 
$$
\left\langle [T_{\varphi}^*,T_{\varphi}]f,f\right\rangle=\left\langle \left|\alpha\right|^2 [T_{\bar{z}},T_{z}]+[T_{{z}^2},T_{\bar{z}^2}]f,f\right\rangle
$$

so that $T_{\varphi}$ is hyponormal if and only if
\begin{equation} \label{equation}
\left|\alpha \right|^2 \left\|zf\right\|^2+\left\langle P(\bar{z}^2f),\bar{z}^2f\right\rangle \ge \left|\alpha\right|^2 \left\langle P(\bar{z}f),\bar{z}f\right\rangle+\left\|z^2f\right\|^2
\end{equation}
for all $f \in A^2(\mathbb D)$. 

A calculation now shows that this happens precisely when $\left|\alpha\right| \ge 2$, as follows. \ For, given $f \in A^2(\mathbb{D}), f \equiv \sum_0^{\infty} b_n z^n$, one can apply Lemma \ref{basiclem} and obtain 
$$
\left\|zf\right\|^2=\sum_0^{\infty} \frac{\left|b_n\right|^2}{n+2},
$$
$$
\left\|P(\bar{z}f)\right\|^2=\sum_1^{\infty} \left|b_n\right|^2 \frac{n}{(n+1)^2},
$$
$$
\left\|z^2f\right\|^2=\sum_0^{\infty} \frac{\left|b_n\right|^2}{n+3},
$$
and
$$
\left\|P(\bar{z}^2f)\right\|^2=\sum_2^{\infty} \left|b_n\right|^2 \frac{n-1}{(n+1)^2}.
$$
Thus, (\ref{equation}) becomes
\begin{equation} \label{eq2}
\left|\alpha\right|^2 \sum_0^{\infty} \frac{\left|b_n\right|^2}{n+2}+\sum_2^{\infty} \left|b_n\right|^2 \frac{n-1}{(n+1)^2} \\
 \ge  \left|\alpha\right|^2 \sum_1^{\infty} \left|b_n\right|^2 \frac{n}{(n+1)^2}+\sum_0^{\infty} \frac{\left|b_n\right|^2}{n+3}.
\end{equation}
In short, equation (\ref{eq2}) must hold for every sequence $(b_n)$ of coefficients of $f$. \ Consider first a sequence $(b_n)$ with $b_0:=1$ and $b_n:=0$ for all $n\ge1$. \ By (\ref{equation}), we have $\left|\alpha\right|^2 \ge \frac{2}{3}$. \ Next, take $b_0:=0, \; b_1:=1$ and $b_n:=0$ for all $n\ge2$; then (\ref{equation}) yields $\left|\alpha\right|^2 \ge 3$. \ Finally, if we fix $k \ge 2$ and we use a sequence $(b_n)$ defined as $b_0:=0, \; b_1:=0, \; \cdots, \; b_{k-1}:=0, \; b_k:=1$ and $b_n:=0$ for all $n > k$, then (\ref{eq2}) becomes
$$
\frac{\left|\alpha\right|^2}{k+2}+\frac{k-1}{(k+1)^2} \ge \left|\alpha\right|^2 \frac{k}{(k+1)^2}+\frac{1}{k+3}.
$$ 
This immediately leads to the condition 
$$
\left|\alpha\right|^2 \ge 4 \cdot \frac{k+2}{k+3} \; \; \; (\textrm{all} \; k \ge 2);
$$
that is, $\left|\alpha\right|^2 \ge 4$. \ As a result, $T_{\varphi}$ is hyponormal if and only if $\left|\alpha\right| \ge 2$. \ It follows that $T_{\bar{z}^2 + 2z}$ is hyponormal.

\subsection{A Key Difference Between the Hardy and Bergman Cases}

\begin{lemma} \label{basiclem}

For $u,v\geq 0$, we have 
\[
P(\bar{z}^{u}z^{v})=\left\{ 
\begin{array}{ccc}
0 &  & v<u \\ 
\frac{(v-u+1)}{v+1}z^{v-u} &  & v\geq u%
\end{array}%
.\right. 
\]
\end{lemma}

\begin{proof}
\begin{eqnarray*}
P(\bar{z}^{u}z^{v}) &=&\sum_{j=0}^{\infty }\left\langle \bar{z}^{u}z^{v},%
\frac{z^{j}}{\left\| z^{j}\right\| }\right\rangle \frac{z^{j}}{\left\|
z^{j}\right\| } \\
&=&\sum_{j=0}^{\infty }\frac{\left\langle \bar{z}^{u}z^{v},z^{j}\right%
\rangle z^{j}}{\left\| z^{j}\right\| ^{2}}=\sum_{j=0}^{\infty
}(j+1)\left\langle z^{v},z^{u+j}\right\rangle z^{j} \\
&=&\left\{ 
\begin{array}{ccc}
0 &  & v<u \\ 
\frac{v-u+1}{v+1}z^{v-u} &  & v\geq u%
\end{array}%
.\right.  
\end{eqnarray*}  \qed
\end{proof}

\begin{corollary}
\label{cor2}
For $v\geq u$ and $t\geq w$, we have 
\begin{eqnarray*}
\left\langle P(\bar{z}^{u}z^{v}),P(\bar{z}^{w}z^{t})\right\rangle
&=&\left\langle \frac{v-u+1}{v+1}z^{v-u},\frac{t-w+1}{t+1}%
z^{t-w}\right\rangle \\
&=&\frac{(t-w+1)}{(v+1)(t+1)}\delta _{u+t,v+w}.
\end{eqnarray*}
\end{corollary}

\subsection{Some Known Results}

In this subsection, we briefly summarize a number of partial results relating to the Bergman space case.

\begin{itemize}

\item (H. Sadraoui \cite{Sad}) \ If $\varphi \equiv \bar g +f$, the following are equivalent: \newline
(i) \ $T_{\varphi}$ is hyponormal on $A^2(\mathbb D)$; \newline
(ii) \ $H_{\bar g}^* H_{\bar g} \le H_{\bar f}^* H_{\bar f}$; \newline
(iii) \ $H_{\bar g} = C H_{\bar f}$, where $C$ is a contraction on $A^2(\mathbb D)$.   

\item (I.S. Hwang \cite{ISH}) \ Let $\varphi \equiv a_{-m} \bar z^m + a_{-N} \bar z^N + a_{m} z^m + a_{N} z^N \; \; (0<m<N)$ satisfying $a_m \bar{a}_N = \bar{a_{-m}}a_{-N}$, then $T_{\varphi}$ is hyponormal if and only if 
\begin{eqnarray*}
\frac{1}{N+1}(\left|a_N\right|^2-\left|a_{-N}\right|^2) & \ge & \frac{1}{m+1}(\left|a_{-m}\right|^2-\left|a_{m}\right|^2) \; \; (\text{if} \left|a_{-N}\right|\le \left|a_{N}\right|) \\
N^2(\left|a_{-N}\right|^2-\left|a_{N}\right|^2) & \le & m^2(\left|a_{m}\right|^2-\left|a_{-m}\right|^2) \; \; (\text{if} \left|a_{N}\right|\le \left|a_{-N}\right|).
\end{eqnarray*}
The last condition is not sufficient.

\item (I.S. Hwang \cite{ISH2}) \ Let $\varphi \equiv 4 \bar{z}^3+2 \bar{z}^2 +\bar{z} + z +2 z^2 + \beta z^3 \; \; (\left|\beta\right|=4)$. \ Then $T_{\varphi}$ is hyponormal if and only if $\beta=4$.

\item (I.S. Hwang \cite{ISH2}) \ Let 
$$
\varphi \equiv 8 \bar{z}^3+ \bar{z}^2 + \beta \bar{z} + \gamma z + 7 z^2 + 2 z^3 \; \; (\left|\beta\right|=\left|\gamma\right|).
$$
Then $T_{\varphi}$ is not hyponormal.

\item (P. Ahern and \v Z. \v Cu\v ckovi\' c \cite{AhCu1}) \ Let $\varphi \equiv \bar g +f \in L^{\infty}(\mathbb D)$, and assume that $T_{\varphi}$ is hyponormal. \ Then 
$$
Bu \ge u,
$$
where $B$ denotes the Berezin transform and $u:=\left|f\right|^2-\left|g\right|^2$.

\item (P. Ahern and \v Z. \v Cu\v ckovi\' c \cite{AhCu1}) \ Let $\varphi \equiv \bar g +f \in L^{\infty}(\mathbb D)$, and assume that $T_{\varphi}$ is hyponormal. \ Then $$
\bar{lim}_{z \rightarrow \zeta} (\left|f'(z)\right|^2 - \left|g'(z)\right|^2) \ge 0
$$
for all $\zeta \in \mathbb T$. \ In particular, if $f'$ and $g'$ are continuous at $\zeta \in \mathbb{T}$ then $\left|f'(\zeta)\right| \ge \left|g'(\zeta)\right|$.

\item (I.S. Hwang and J. Lee \cite{ISHJL}) \ The authors obtain some basic results on hyponormality of Toeplitz operators on weighted Bergman spaces.

\item (Y. Lu and C. Liu \cite{LuLiu}) \ The authors obtain necessary and sufficient conditions for the hyponormality of $T_{\varphi}$ in the case when $\varphi$ is a radial symbol.

\item (Y. Lu and Y. Shi \cite{LuShi}) \ The authors study the weighted Bergman space case, and prove the following result.
\begin{theorem}
(cf. \cite[Theorem 2.4(ii)]{LuShi}) \ Let $\varphi:=\alpha z^n + \beta z^m + \gamma \bar{z}^m + \delta \bar{z}^n$, with $n > m$. \ Then $m^2(\left| \beta \right| ^2 - \left| \gamma \right|^2) + n^2(\left| \alpha \right| ^2 - \left| \delta \right| ^2) \ge mn \left|\bar{\alpha}\beta-\bar{\gamma}\delta \right|$.
\end{theorem}

\end{itemize}

\subsection{Hyponormality of Toeplitz Operators on the Bergman Space}

The self-commutator of $T_{\varphi}$ is  

$$C:=[T^*_{\varphi},T_{\varphi}]$$

We seek necessary and sufficient conditions on the symbol $\varphi$ to ensure that $C \ge 0$.

The following result gives a flavor of the type of calculations we face when trying to decipher the hyponormality of a Toeplitz operator acting on the Bergman space. \ Although the calculation therein will be superseded by the calculations in the following section, it serves both as a preliminary example and as  motivation for the organization of our work.
  
\begin{proposition}
Assume $k, \ell \geq \max\{a,b\}$. \ Then
\begin{eqnarray*}
&&\left\langle \left[ T_{\bar{z}^{a}},T_{z^{b}}\right] (z^{k}+cz^{\ell
}),z^{k}+cz^{\ell }\right\rangle  \\
&=&a^{2}\left[ \frac{1}{(k+1)^{2}(k+1+a)}+c^{2}\cdot \frac{1}{(\ell
+1)^{2}(\ell +1+a)}\right] \delta _{a,b} \\
&&+ac[\frac{k-\ell +a}{\left( a+k+1\right) \left( k+1\right) \left( \ell
+1\right) }\delta _{a+k,b+\ell } \\
&&+\frac{\ell -k+a}{\left( a+\ell +1\right) \left( k+1\right) \left( \ell
+1\right) }\delta _{a+\ell ,b+k}]
\end{eqnarray*}
\end{proposition}

\begin{proof}
Keeping in mind that $k, \ell \geq \max\{a,b\}$, we calculate the action of the commutator on the binomial $z^{k}+cz^{\ell}$.

\begin{eqnarray*}
&&\left\langle \left[ T_{\bar{z}^{a}},T_{z^{b}}\right] (z^{k}+cz^{\ell
}),z^{k}+cz^{\ell }\right\rangle  \\
&=&\left\langle T_{\bar{z}^{a}}T_{z^{b}}(z^{k}+cz^{\ell }),z^{k}+cz^{\ell
}\right\rangle -\left\langle T_{z^{b}}T_{\bar{z}^{a}}(z^{k}+cz^{\ell
}),z^{k}+cz^{\ell }\right\rangle  \\
&=&\left\langle z^{b+k}+cz^{b+\ell },z^{a+k}+cz^{a+\ell }\right\rangle
-\left\langle P(\bar{z}^{a}z^{k}+c\bar{z}^{a}z^{\ell }),P(\bar{z}^{b}z^{k}+c%
\bar{z}^{b}z^{\ell })\right\rangle  \\
&=&\frac{\delta _{a,b}}{a+k+1}+c\frac{\delta _{a+k,b+\ell }}{a+k+1}+c\frac{%
\delta _{a+\ell ,b+k}}{a+\ell +1}+c^{2}\frac{\delta _{a,b}}{a+\ell +1} \\
&&-\frac{(k-b+1)}{(k+1)^{2}}\delta _{a,b}-c\frac{(k-a+1)}{(k+1)(\ell +1)}%
\delta _{a+\ell ,b+k} \\
&&-c\frac{(k-b+1)}{(k+1)(\ell +1)}\delta _{a+k,b+\ell }-c^{2}\frac{(\ell
-b+1)}{(\ell +1)^{2}}\delta _{a,b} \\
&=&\left[ \frac{1}{a+k+1}+c^{2}\cdot \frac{1}{a+\ell +1}-\frac{(k-b+1)}{%
(k+1)^{2}}-c^{2}\cdot \frac{(\ell -b+1)}{(\ell +1)^{2}}\right] \delta _{a,b}
\\
&&+c\left[ \frac{1}{a+k+1}-\frac{(k-b+1)}{(k+1)(\ell +1)}\right] \delta
_{a+k,b+\ell } \\
&&+c\left[ \frac{1}{a+\ell +1}-\frac{(k-a+1)}{(k+1)(\ell +1)}\right] \delta
_{a+\ell ,b+k} \\
&=&\left[ \frac{(k+1)(b-a)+ab}{(k+1)^{2}(k+1+a)}+c^{2}\cdot \frac{(\ell
+1)(b-a)+ab}{(\ell +1)^{2}(\ell +1+a)}\right] \delta _{a,b} \\
&&+c\cdot \left[ \frac{1}{a+k+1}-\frac{(\ell -a+1)}{(k+1)(\ell +1)}\right]
\delta _{a+k,b+\ell } \\
&&+c\left[ \frac{1}{a+\ell +1}-\frac{(k-a+1)}{(k+1)(\ell +1)}\right] \delta
_{a+\ell ,b+k} \\ 
&=&\left[ \frac{a^{2}}{(k+1)^{2}(k+1+a)}+c^{2}\cdot \frac{a^{2}}{(\ell
+1)^{2}(\ell +1+a)}\right] \delta _{a,b} \\
&&(if \ a\neq b \ then \ the \ whole \ expression \ is \ 0) \\
&&+ac\cdot \frac{k-\ell +a}{\left( a+k+1\right) \left( k+1\right) \left(
\ell +1\right) }\delta _{a+k,b+\ell } \\
&&+ac\cdot \frac{\ell -k+a}{\left( a+\ell +1\right) \left( k+1\right) \left(
\ell +1\right) }\delta _{a+\ell ,b+k} 
\end{eqnarray*}
\begin{eqnarray*}
&=&a^{2}\left[ \frac{1}{(k+1)^{2}(k+1+a)}+c^{2}\cdot \frac{1}{(\ell
+1)^{2}(\ell +1+a)}\right] \delta _{a,b} \\
&&+ac[\frac{k-\ell +a}{\left( a+k+1\right) \left( k+1\right) \left( \ell
+1\right) }\delta _{a+k,b+\ell } \\
&&+\frac{\ell -k+a}{\left( a+\ell +1\right) \left( k+1\right) \left( \ell
+1\right) }\delta _{a+\ell ,b+k}],
\end{eqnarray*}
as desired.  \qed
\end{proof}

\begin{corollary} \label{cor26}
Assume $a=b$, $k, \ell \geq a$ and $k \ne \ell$. \ Then
\begin{eqnarray*}
&&\left\langle \left[ T_{\bar{z}^{a}},T_{z^{a}}\right] (z^{k}+cz^{\ell
}),z^{k}+cz^{\ell }\right\rangle  \\
&=&a^{2}\left[ \frac{1}{(k+1)^{2}(k+1+a)}+c^{2}\cdot \frac{1}{(\ell
+1)^{2}(\ell +1+a)}\right] .
\end{eqnarray*}
\end{corollary}

\subsection{Self-Commutators}

We focus on the action of the self-commutator $C$ of certain Toeplitz operators $T_{\varphi}$ on suitable vectors $f$ in the space $A^2(\mathbb{D})$. \ The symbol $\varphi$ and the vector $f$ are of the form 
$$
\varphi:=\alpha z^n+\beta z^m +\gamma \overline z ^p + \delta \overline z ^q \; \; (n>m; \; p<q)
$$ 
and 
$$
f:=z^k+c z^{\ell}+d z^r \; \; (k<\ell<r),
$$
respectively, with $\ell$ and $r$ to be determined later. \ We also assume that $n-m=q-p$. \ Our ultimate goal is to study the asymptotic behavior of this action as $k$ goes to infinity. \ Thus, we consider the expression $\left\langle Cf,f\right\rangle$, given by
$$
\left\langle [(T_{\alpha z^n+\beta z^m +\gamma \overline z ^p + \delta \overline z ^q})^*,T_{\alpha z^n+\beta z^m +\gamma \overline z ^p + \delta \overline z ^q}](z^k+c z^{\ell}+d z^r),z^k+c z^{\ell}+d z^r\right\rangle ,
$$
for large values of $k$ (and consequently large values of $\ell$ and $r$.) \ It is straightforward to see that $\left\langle Cf,f\right\rangle$ is a quadratic form in $c$ and $d$, that is, 

\begin{equation}
\label{quadraticform}
\left\langle Cf,f\right\rangle \equiv A_{00}+2 \; \text{Re} (A_{10} c) + 2 \; \text{Re} (A_{01} d) + A_{20}c \overline{c} + 2 \; \text{Re} (A_{11} 
\overline{c}d)+ A_{02} d \overline{d} .
\end{equation}

Alternatively, the matricial form of (\ref{quadraticform}) is
\begin{equation}
\label{matrixform}
\left\langle 
\left(
\begin{array}{lcr}
A_{00} & A_{10} & A_{01} \\
\overline A_{10} & A_{20} & A_{11} \\
\overline A_{01} & \overline A_{11} & A_{02}
\end{array}
\right)
\left(
\begin{array}{l}
1 \\
c \\
d
\end{array}
\right) ,
\left(
\begin{array}{l}
1 \\
c \\
d
\end{array}
\right)
\right\rangle .
\end{equation} 

We now observe that the coefficient $A_{00}$ arises from the action of $C$ on the monomial $z^k$, that is,
$$
A_{00} = \left\langle Cz^k,z^k \right\rangle \equiv \left\langle [(T_{\alpha z^n+\beta z^m +\gamma \overline z ^p + \delta \overline z ^q})^*,T_{\alpha z^n+\beta z^m +\gamma \overline z ^p + \delta \overline z ^q}]z^k,z^k \right\rangle .
$$
Similarly, 
$$
A_{10} = \left\langle Cz^{\ell},z^k \right\rangle ,
$$
\begin{equation} \label{a01}
A_{01} = \left\langle Cz^r,z^k \right\rangle ,
\end{equation}
$$
A_{20} = \left\langle Cz^{\ell},z^{\ell} \right\rangle ,
$$
\begin{equation} \label{a11}
A_{11} = \left\langle Cz^r,z^{\ell} \right\rangle ,
\end{equation}
$$
A_{02} = \left\langle Cz^r,z^r \right\rangle .
$$

To calculate $A_{00}$ explicitly, we first recall that the algebra of Toeplitz operators with analytic symbols is commutative, and therefore $T_{z^n}$ commutes with $T_{z^m}$, $T_{z^p}$ and $T_{z^q}$. \ 

We also recall that two monomials $z^u$ and $z^v$ are orthogonal whenever $u \ne v$. \ As a result, the only nonzero contributions to $A_{00}$ must come from the inner products $\left\langle [T_{z^n}^*,T_{z^n}]z^k,z^k \right\rangle$, $\left\langle [T_{z^m}^*,T_{z^m}]z^k,z^k \right\rangle$, $\left\langle [T_{z^p}^*,T_{z^p}]z^k,z^k \right\rangle$ and $\left\langle [T_{z^q}^*,T_{z^q}]z^k,z^k \right\rangle$ . \ 

Applying Corollary \ref{cor26} we see that
$$
A_{00}=\frac{1}{(k+1)^2}\left(\frac{\left|\alpha\right|^2n^2}{k+n+1}+\frac{\left|\beta\right|^2m^2}{k+m+1}-\frac{\left|\gamma\right|^2p^2}{k+p+1}-\frac{\left|\delta\right|^2q^2}{k+q+1}\right) .
$$

Similarly, 
\begin{eqnarray}
A_{10}&=&\overline{\alpha}\beta(\frac{1}{\ell+m+1}-\frac{k-m+1}{(k+1)(\ell+1)})\delta_{n+k,m+\ell} \nonumber \\
&&+\alpha\overline{\beta}(\frac{1}{\ell+n+1}-\frac{k-n+1}{(k+1)(\ell+1)})\delta_{m+k,n+\ell} \nonumber \\
&&-\overline{\gamma}\delta(\frac{1}{\ell+p+1}-\frac{k-p+1}{(k+1)(\ell+1)})\delta_{q+k,p+\ell} \nonumber \\
&&-\gamma\overline{\delta}(\frac{1}{\ell+q+1}-\frac{k-q+1}{(k+1)(\ell+1)})\delta_{p+k,q+\ell} \label{A10}.
\end{eqnarray}
Now recall that $m<n$ and $k<\ell$, so that $m+k<n+\ell$, and therefore $\delta_{m+k,n+\ell}$=0. \ Also, $p<q$ implies $p+k<q+\ell$, so that $\delta_{p+k,q+\ell}=0$. \ As a consequence, 
\begin{eqnarray}
A_{10}&=&\overline{\alpha}\beta(\frac{1}{\ell+m+1}-\frac{k-m+1}{(k+1)(\ell+1)})\delta_{n+k,m+\ell} \nonumber \\
&&-\overline{\gamma}\delta(\frac{1}{\ell+p+1}-\frac{k-p+1}{(k+1)(\ell+1)})\delta_{q+k,p+\ell} .
\end{eqnarray}

Consider now $A_{01}$, as described in (\ref{a01}). \ We wish to imitate the calculation for $A_{10}$. \ Observe that $k<r$, so that the vanishing of the relevant $\delta$'s in (\ref{A10}) still holds. \ Thus, we obtain
\begin{eqnarray} \label{A01}
A_{01}&=&\overline{\alpha}\beta(\frac{1}{r+m+1}-\frac{k-m+1}{(k+1)(r+1)})\delta_{n+k,m+r} \nonumber \\
&&-\overline{\gamma}\delta(\frac{1}{r+p+1}-\frac{k-p+1}{(k+1)(r+1)})\delta_{q+k,p+r} .
\end{eqnarray}
In short, $A_{01}$ can be obtained from $A_{10}$ by replacing $\ell$ by $r$. \ In a completely analogous way, we can calculate $A_{11}$, by replacing $\ell$ by $r$ and $k$ by $\ell$ in (\ref{A10}): 
\begin{eqnarray}
A_{11}&=&\overline{\alpha} \beta (\frac{1}{r+m+1}-\frac{\ell-m+1}{(\ell+1)(r+1)})\delta_{n+\ell,m+r} \nonumber \\
&&-\overline{\gamma} \delta (\frac{1}{r+p+1}-\frac{\ell-p+1}{(r+1)(\ell+1)})\delta_{q+\ell,p+r} .
\end{eqnarray}
Also, $A_{20}$ and $A_{02}$ follow the pattern of $A_{00}$:
\begin{eqnarray*}
A_{20}&=&\frac{1}{(\ell+1)^2}\left(\frac{\left|\alpha\right|^2n^2}{\ell+n+1}+\frac{\left|\beta\right|^2m^2}{\ell+m+1}-\frac{\left|\gamma\right|^2p^2}{\ell+p+1}-\frac{\left|\delta\right|^2q^2}{\ell+q+1}\right) ,
\end{eqnarray*}
and
\begin{eqnarray*}
A_{02}&=&\frac{1}{(r+1)^2}\left(\frac{\left|\alpha\right|^2n^2}{r+n+1}+\frac{\left|\beta\right|^2m^2}{r+m+1} -\frac{\left|\gamma\right|^2p^2}{r+p+1}-\frac{\left|\delta\right|^2q^2}{r+q+1}\right) .
\end{eqnarray*}

Recall again that $k<\ell<r$. \ We now make a judicious choice to simplify the forms of $A_{10}$, $A_{11}$ and $A_{01}$. \ That is, we let $\ell:=n+k-m$ and $r:=\ell+q-p$. \ It follows that $n+k=m+\ell<m+r$ and $q+k<q+\ell=p+r$. \ Therefore, both Kronecker deltas appearing in $A_{01}$ are zero, and thus $A_{01}=0$. \ Moreover, 
\begin{eqnarray} \label{newA10}
A_{10}&=&\overline{\alpha}\beta(\frac{1}{\ell+m+1}-\frac{k-m+1}{(k+1)(\ell+1)}) \\
&&-\overline{\gamma}\delta(\frac{1}{\ell+p+1}-\frac{k-p+1}{(k+1)(\ell+1)})\delta_{q+k,p+\ell} .  
\end{eqnarray}
and
\begin{eqnarray*}
A_{11}&=&\overline{\alpha}\beta (\frac{1}{r+m+1}-\frac{\ell-m+1}{(r+1)(\ell+1)})\delta_{n+\ell,m+r} \\
&&-\bar{\gamma} \delta (\frac{1}{r+p+1}-\frac{\ell-p+1}{(r+1)(\ell+1)}) .
\end{eqnarray*}
The $3 \times 3$ matrix associated with $C$ becomes
$$
M:=\left(
\begin{array}{lcr}
A_{00} & A_{10} & 0 \\
\overline A_{10} & A_{20} & A_{11} \\
0 & \overline A_{11} & A_{02}
\end{array}
\right) .
$$
We now wish to study the asymptotic behavior of $k^3 M$ as $k \rightarrow \infty$. \ Surprisingly, $k^3A_{00}$, $k^3A_{02}$ and $k^3A_{20}$ all have the same limit as $k \rightarrow \infty$. \ Also, $k^3A_{10}$ and $k^3\overline A_{11}$ have the same limit. \ To see this, observe that
$$
k^3A_{00}=\frac{k^2}{(k+1)^2}\left(\frac{k\left|\alpha\right|^2n^2}{k+n+1}+\frac{k\left|\beta\right|^2m^2}{k+m+1}-\frac{k\left|\gamma\right|^2p^2}{k+p+1}-\frac{k\left|\delta\right|^2q^2}{k+q+1}\right) ,
$$
so that 
$$
a:=\lim_{k \rightarrow \infty} k^3A_{00}=\left|\alpha\right|^2n^2 + \left|\beta\right|^2m^2 - \left|\gamma\right|^2p^2 - \left|\delta\right|^2q^2.
$$
Then $\lim_{k \rightarrow \infty} k^3A_{20}=\lim_{k \rightarrow \infty} k^3A_{02}=a$.
In terms of the remaining entries of $k^3M$, recall the assumption $n-m=q-p$, and let $g:=n-m=q-p$. \ It follows that $\ell=k+g$ and $r=\ell+g=k+2g$. \ By using these values in (\ref{newA10}), we obtain
$$
k^3A_{10}=\bar{\alpha} \beta \frac{k^3 mn}{(k+1)(k+g+1)(k+g+m+1)} - \bar{\gamma} \delta \frac{k^3 pq}{(k+1)(k+g+1)(k+g+p+1)} ,
$$
so that 
$$
\rho:=\lim_{k \rightarrow \infty} k^3A_{10}=\bar{\alpha} \beta mn - \bar{\gamma} \delta pq .
$$
The calculation for $A_{11}$ is entirely similar, and one gets $\lim_{k \rightarrow \infty} k^3 A_{11}=\rho$.

It follows that the asymptotic behavior of $k^3M$ is given by the tridiagonal matrix
$$
\left(
\begin{array}{lcr}
a & \rho & 0 \\
\bar{\rho} & a & \rho \\
0 & \bar{\rho} & a
\end{array}
\right) .
$$
 
Now, if instead of using a vector of the form
$$
f:=z^k+c z^{\ell}+d z^r \; \; (k<\ell<r),
$$
with $\ell=k+g$ and $r=\ell+g=k+2g$ (that is, a vector of the form
$$
f:=z^k+c z^{k+g}+d z^{k+2g}
$$
we were to use a longer vector with similar power structure,
$$
f_N:=z^k+c_1 z^{k+g}+c_2 z^{k+2g}+ \cdots + c_N z^{k+Ng},
$$
it is not hard to see that the asymptotic behavior of the associated matrix would still be given by the tridiagonal matrix with $a$ in the diagonal and $\rho$ in the super-diagonal. \ To see this, one only need to observe that the entries of the matrix $P$ associated with $\left\langle Cf,f \right\rangle$ will follow the same pattern as the entries in $M$. \ For example, when $N=3$ the $(3,4)$-entry of $P$ will follow the pattern of $A_{10}$ above, with $\ell$ and $k$ replaced by $k+3g$ and $k+2g$, resp. \ Similarly, the $(2,4)$-entry of $P$ will follow the pattern of $A_{01}$ above, with $r$ and $k$ replaced by $k+3g$ and $k+g$, resp. \ As a result, it is straightforward to see that, like $A_{01}$, the entry $P_{24}$ will be zero. \ As for $P_{34}$, one gets
\begin{eqnarray}
P_{34}&=&\overline{\alpha} \beta (\frac{1}{k+3g+m+1}-\frac{k+2g-m+1}{(k+2g+1)(k+3g+1)})\delta_{n+k+2g,m+k+3g} \nonumber \\
&&-\overline{\gamma} \delta (\frac{1}{k+3g+p+1}-\frac{k+2g-p+1}{(k+3g+1)(k+2g+1)})\delta_{q+k+2g,p+k+3g} . \nonumber \\
&& 
\end{eqnarray}
As before, 
\begin{eqnarray*}
k^3P_{34}&=&\bar{\alpha} \beta \frac{k^3 mn}{(k+2g+1)(k+3g+1)(k+3g+m+1)} \\
&&- \bar{\gamma} \delta \frac{k^3 pq}{(k+2g+1)(k+3g+1)(k+3g+p+1)} ,
\end{eqnarray*}
and once again, 
$$
\lim_{k \rightarrow \infty} k^3P_{34}=\bar{\alpha} \beta mn - \bar{\gamma} \delta pq = \rho.
$$

In summary, the hyponormality of $T_{\varphi}$, detected by the positivity of the self-commutator $C$, leads to the positive semi-definiteness of the tridiagonal $(N+1) \times (N+1)$ matrix $P$. \ Since this must be true for all $N \ge 1$, it follows that a necessary condition for the hyponormality of $T_{\varphi}$ is the positive semi-definiteness of the infinite tridiagonal matrix 
$$
Q:=\left(
\begin{array}{lccr}
a & \rho & 0 & \cdots \\
\bar{\rho} & a & \rho & \cdots \\
0 & \bar{\rho} & a & \cdots \\
\vdots & \vdots & \vdots & \ddots 
\end{array}
\right) .
$$

We now consider the spectral behavior of $Q$ as an operator on $\ell^2(\mathbb{Z}_+)$.   

\begin{lemma}

For $a \in \mathbb{R}$ and $\rho \in \mathbb{C}$, the spectrum of the infinite tridiagonal matrix $Q$ is $[a-2\left|\rho\right|,a+2\left|\rho\right|]$.
\end{lemma}

\begin{proof}
This result is well known; we present a proof for the sake of completeness. \ Observe that $Q$ is the canonical matrix representation of the Toeplitz operator on $H^2(\mathbb{T})$ with symbol $\varphi(z):=a+2 \textrm{Re} (\bar{\rho}z)$. \ Since the symbol is harmonic, it follows that the spectrum of $T_{\varphi} \equiv aI+T_{\bar{\rho}z+\rho \bar{z}}$ is the set $a+2 \textrm{Re} ({\left\{\bar{\rho}z: \; z \in \mathbb{D} \right\}}^{-}) = a+2[-\left|\rho\right|,\left|\rho\right|]$, as desired. \qed
\end{proof}
 
As a consequence, if $Q$ is positive (as an operator on $\ell^2(\mathbb Z_+)$), then 
$$
a \ge 2\left|\rho\right|.
$$ 

\subsection{Main Result}

\begin{theorem} \label{mainthm}
Assume that $T_{\varphi}$ is hyponormal on $A^2(\mathbb D)$, with
$$
\varphi:=\alpha z^n+\beta z^m +\gamma \overline z ^p + \delta \overline z ^q \; \; (n>m; \; p<q).
$$
Assume also that $n-m=q-p$. \ Then
\begin{equation} \label{eqthm}
\left|\alpha \right|^2 n^2 + \left|\beta \right|^2 m^2 - \left|\gamma \right|^2 p^2 - \left|\delta \right|^2 q^2  \ge 2 \left|\bar \alpha \beta m n - \bar \gamma \delta p q \right|.
\end{equation}
\end{theorem}

\subsection{A Specific Case}

When $p=m$ and $q=n$ in
$$
\varphi:=\alpha z^n+\beta z^m +\gamma \overline z ^p + \delta \overline z ^q \; \; (n>m; \; p<q).
$$
the inequality
$$
\left|\alpha \right|^2 n^2 + \left|\beta \right|^2 m^2 - \left|\gamma \right|^2 p^2 - \left|\delta \right|^2 q^2  \ge 2 \left|\bar \alpha \beta m n - \bar \gamma \delta p q \right|.
$$
reduces to
$$
n^2(\left|\alpha \right|^2 - \left|\delta \right|^2) + m^2 (\left|\beta \right|^2 - \left|\gamma \right|^2) \ge 2 mn \left|\bar \alpha \beta - \bar \gamma \delta \right|.
$$
This not only generalizes previous estimates, but also sharpens them, since previous results did not include the factor $2$ in the right-hand side.

\section{When is $T_{\varphi}$ normal?}

We conclude this paper with a description of those symbols $\varphi$ in Theorem \ref{mainthm} which produce a normal operator $T_{\varphi}$. \ We first recall a result of S. Axler and \v Z. \v Cu\v ckovi\' c.

\begin{lemma} (\cite{AxCu}) \label{normal} \ Let $\varphi$ be harmonic and bounded on $\mathbb{D}$. \ Then $T_{\varphi}$ is normal if and only if there exist a pair of complex numbers $a$ and $b$ such that $(a,b) \ne (0,0)$ and $F:=a \varphi + b \bar{\varphi}$ is constant on $\mathbb{D}$.
\end{lemma}

Assume now that $T_{\varphi}$ is normal. \ By Lemma \ref{normal}, there exist $a$ and $b$ such that $(a,b) \ne (0,0)$ and $F:=a \varphi + b \bar{\varphi}$ is constant. \ In what follows, we write a harmonic symbol as $\varphi \equiv f+\bar{g}$, with $f$ and $g$ analytic. \ A straightforward calculation using 
$\frac{\partial}{\partial z}$ and $\frac{\partial}{\partial \bar{z}}$, applied to $F$, shows that $(\left|a\right|^2-\left|b\right|^2)\frac{\partial f}{\partial z}=0$. \ If $f$ is constant, a similar calculation shows that $g$ is also constant, and a fortiori $\varphi$ is constant. \ Thus, without loss of generality, we can assume that $f$ is not constant, and therefore $\left|a\right|=\left|b\right|>0$. \ If we write $a=\left|a\right|e^{i\theta}$ and $b=\left|a\right|e^{i\eta}$, it is not hard to see that $\varphi + e^{i(\eta-\theta)}\bar{\varphi}$ is constant on $\mathbb{D}$. \ Let $\lambda:= e^{{i}(\eta-\theta)}$, so that $\left| \lambda \right| =1$. \ We conclude that $\varphi+\lambda \overline{\varphi}$ is constant on $\mathbb{D}$.      

\begin{theorem} \label{normalthm} 
Let $\varphi \equiv \alpha z^n + \beta z^m + \gamma \bar{z}^p + \delta \bar{z}^q$, with $n<m$, $p<q$, $n-m=q-p$ and $(\alpha,\beta,\gamma,\delta) \ne (0,0,0,0)$. \ Then $T_{\varphi}$ is normal if and only if there exists $\lambda \in \mathbb{T}$ such that $\varphi$ is of one of exactly three types: \newline
(i) $\varphi=\alpha z^n - \lambda \bar{\alpha} \bar{z}^n$ (when $n=p$); \newline
(ii) $\varphi=\alpha z^n +\beta z^m - \lambda (\bar{\alpha} \bar{z}^n+\bar{\beta} \bar{z}^m)$ (when $m=p$); or \newline
(iii) $\varphi=\beta z^m - \lambda \bar{\beta} \bar{z}^m$ (when $m=q$).  
\end{theorem}

\begin{proof}
($\Longrightarrow$) \ Assume that $T_{\varphi}$ is normal. \ From the discussion in the paragraph immediately preceding Theorem \ref{normalthm}, we can always assume that $\varphi+\lambda \overline{\varphi}$ is constant on $\mathbb{D}$, for some $\lambda \in \mathbb{T}$. \ Since $\varphi$ is clearly nonconstant, we know that 
$G:=\varphi+\lambda \overline{\varphi}$ is a constant trigonometric polynomial, with analytic monomials $z^m$, $z^n$, $z^p$ and $z^q$. \ Since $\varphi$ is a nonconstant harmonic function, in the above mentioned list of four monomials we must necessarily have at least two identical monomials. \ Since $m<n$, $p<q$ and $n-m=q-p$, we are led to consider the following three cases: \newline
(i) $n=p$ (and therefore $m<n=p<q$); here 
$$
G=\beta z^m+(\alpha+\lambda \bar{\gamma})z^n+\lambda \bar{\delta}\bar{z}^q+\lambda \overline{\beta z^m+(\alpha+\lambda \bar{\gamma})z^n+\lambda \bar{\delta}z^q},
$$
from which it easily follows that $\beta=0$, $\gamma=-\lambda \bar{\alpha}$ and $\delta=0$. \ Then $\varphi=\alpha z^n-\lambda \overline{\alpha z^n}$, as desired. \newline
(ii) $m=p$ (and therefore $m=p<q=n$); here 
$$
G=(\alpha + \lambda \bar{\delta}) z^n+(\beta + \lambda \bar{\gamma})z^m + \lambda \overline{(\alpha + \lambda \bar{\delta}) z^n +(\beta + \lambda \bar{\gamma})z^m},
$$
so that $\alpha+\lambda \bar{\delta}=0$ and $\beta+\lambda \bar{\gamma}=0$. It readily follows that $\delta = - \lambda \bar{\alpha}$ and $\gamma = - \lambda \bar{\beta}$. \ We then get $\varphi=\alpha z^n+\beta z^m - \lambda \overline{\alpha z^n+\beta z^m}$, as desired. \newline
(iii) $m=q$, which leads to $\varphi=\beta z^m-\lambda \overline{\beta z^m}.$ \newline
($\Longleftarrow$) \ For the converse, observe that in each of the three representations we have $\varphi+\lambda \bar{\varphi}=0$, which implies $T_{\varphi}^*=-\bar{\lambda} T_{\varphi}$. \ Therefore, $T_{\varphi}^*$ commutes with $T_{\varphi}$, so $T_{\varphi}$ is normal. \newline
The proof is now complete. \qed
\end{proof}

\begin{remark}
The form of (i), (ii) and (iii) in Theorem \ref{normalthm} is entirely consistent with Theorem \ref{mainthm}. \ For instance, consider case (i); here $\beta=\delta=0$ and $\gamma=-\lambda \overline{\alpha}$, so that both sides of (\ref{eqthm}) are equal to $0$.
\end{remark} 



\begin{thebibliography}{99}

\bibitem{AhCu1} P. Ahern and \v Z. \v Cu\v ckovi\' c, A mean value inequality with applications to Bergman space operators, Pacific J. Math. \textbf{173}(1996), 295--305.



\bibitem{AxCu} S. Axler and \v Z. \v Cu\v ckovi\' c, Commuting Toeplitz operators with harmonic symbols, Integral Equations Operator Theory \textbf{14}(1991), 1--12.

\bibitem{Co}  C. Cowen, Hyponormality of Toeplitz operators, Proc. Amer. Math. Soc. \textbf{103}(1988), 809--812.

\bibitem{CoL}  C. Cowen and J. Long, Some subnormal Toeplitz operators, J. Reine Angew. Math. \textbf{351}(1984), 216--220.


\bibitem{CHL} R.E. Curto, I.S. Hwang and W.Y. Lee, \textit{Matrix Functions of Bounded Type: An Interplay Between Function Theory and Operator Theory}, preprint 2016; submitted to Memoirs Amer. Math. Soc.

\bibitem{CL1}  R.E. Curto and W.Y. Lee, \textit{Joint hyponormality of Toeplitz pairs}, Memoirs Amer. Math. Soc. \textbf{712}, Amer. Math. Soc., Providence, 2001.


\bibitem{FaLe} D. Farenick and W.Y. Lee, Hyponormality and spectra of Toeplitz operators, Trans. Amer. Math. Soc. \textbf{348}(1996), 4153--4174.

\bibitem{Gu} C. Gu, A generalization of Cowen's characterization of hyponormal Toeplitz operators, J. Funct. Anal. \textbf{124}(1994), 135--148.

\bibitem{Hal1} P.R. Halmos, Ten problems in Hilbert space, Bull. Amer. Math. Soc. \textbf{76}(1970), 887--933.


\bibitem{ISH} I.S. Hwang, Hyponormal Toeplitz operators on the Bergman space, J. Korean Math. Soc. \textbf{42}(2005), 387--403.

\bibitem{ISH2} I.S. Hwang, Hyponormality of Toeplitz operators on the Bergman space, J. Korean Math. Soc. \textbf{45}(2008), 1027--1041.


\bibitem{HL} I.S. Hwang and W.Y. Lee, Hyponormality of trigonometric Toeplitz operators, Trans. Amer. Math. Soc. \textbf{354}(2002), 2461--2474.

\bibitem{ISHJL1} I.S. Hwang and J. Lee, Hyponormal Toeplitz operators on the Bergman space, Bull. Korean Math. Soc. \textbf{44}(2007), 517--522.

\bibitem{ISHJL} I.S. Hwang and J. Lee, Hyponormal Toeplitz operators on the weighted Bergman spaces, Math. Ineq. Appl. \textbf{15}(2012), 323--330.

\bibitem{LuLiu} Y. Lu and C. Liu, Commutativity and hyponormality of Toeplitz operators on the weighted Bergman space, J. Korean Math. Soc. 
\textbf{46}(2009), 621--642.

\bibitem{LuShi} Y. Lu and Y. Shi, Hyponormal Toeplitz operators on the weighted Bergman space, Integral Equations Operator Theory \textbf{65}(2009), 115--129.

\bibitem{NT}  T. Nakazi and K. Takahashi, Hyponormal Toeplitz operators and extremal problems of Hardy spaces, Trans. Amer. Math. Soc. \textbf{338}(1993), 753--769.

\bibitem{Sad} H. Sadraoui, \textit{Hyponormality of Toeplitz Operators and Composition Operators}, Ph.D. Dissertation, Purdue University, 1992.

\bibitem{Schur} I. Schur, \"Uber Potenzreihen die im Innern des Einheitskreises beschr\"ankt sind, J. Reine Angew. Math. \textbf{147}(1917), 205--232.
 
\bibitem{Zhu} K. Zhu, Hyponormal Toeplitz operators with polynomial symbols, Integral Equations Operator Theory \textbf{21}(1995), 376--381.

\end{thebibliography}
\end{document}